\documentclass[a4paper]{smfart}

\usepackage[utf8]{inputenc}
\usepackage[T1]{fontenc}
\usepackage[francais, english]{babel}
\usepackage{mathrsfs}
\usepackage{amssymb,amsmath,amsthm}
\usepackage{url}

\usepackage{graphicx}

\newcommand{\f}{\mathbb{F}}
\newcommand{\z}{\mathbb{Z}}
\newcommand{\tfjm}{$\mathbb{TFJM}^2$}

\newtheorem{thm}{Théorème}
\newtheorem{lem}[thm]{Lemme}
\newtheorem{pro}[thm]{Proposition}
\newtheorem*{rem}{Remarque}

\begin{document}

\selectlanguage{french}

\title{Un r\'esultat de recherche obtenu gr\^ace \`a des lyc\'eens}
\author{Thomas Budzinski \and Giancarlo Lucchini Arteche}
\date{}

\thanks{Nous tenons \`a remercier Ramon Moreira Nunes pour avoir attir\'e notre attention sur cette question pendant l'organisation du \tfjm\, 2014, ainsi que pour avoir servi de contact avec Emmanuel Kowalski par la suite. Nous remercions aussi ce dernier pour son soutien et pour avoir point\'e les autres preuves existantes.}

\maketitle

\begin{abstract}
Nous pr\'esentons une d\'emonstration \'el\'ementaire d'une question pos\'ee par Fouvry, Kowalski et Michel sur les solutions de l'\'equation $x+\frac{1}{x}+y+\frac{1}{y}+t=0$ sur un corps fini, en utilisant des r\'esultats obtenus par des lyc\'eens lors de leur participation au Tournoi Fran\c cais des Jeunes Math\'ematiciennes et Math\'ematiciens.
\end{abstract}

\section{Introduction}
Le Tournoi Fran\c cais des Jeunes Math\'ematiciennes et Math\'ematiciens \tfjm\, est un tournoi de math\'ematiques destin\'e aux \'el\`eves de lyc\'ee qui existe depuis 2011. Il se distingue d'autres comp\'etitions comme les Olympiades de math\'ematiques en proposant des probl\`emes ouverts dont les \'enonc\'es sont connus \`a l'avance et en \'etant organis\'e par \'equipes. Guid\'ees par des encadrants, les \'equipes compos\'ees de 4, 5 ou 6 \'el\`eves ont environ trois mois pour r\'efl\'echir aux probl\`emes propos\'es. Lors du tournoi, les différentes \'equipes pr\'esentent leurs r\'esultats et en discutent avec les autres \'equipes sous la forme de ``joutes’’ math\'ematiques (voir l’appendice par Matthieu Lequesne et le deuxi\`eme auteur). Les probl\`emes propos\'es sont inhabituels pour la plupart des \'el\`eves, car ils n'admettent pas, \`a la connaissance du jury, de solution compl\`ete. Pour les \'equipes, il s'agit donc de comprendre le probl\`eme, de r\'esoudre des cas particuliers, de rep\'erer les difficult\'es\ldots \\

Le but de ce texte est de montrer \`a quel point le \tfjm\, peut amener les \'el\`eves de lyc\'ee à entreprendre une d\'emarche de recherche scientifique. Cette histoire date de l’\'edition 2014 et commence par un post dans le blog math\'ematique d'Emmanuel Kowalski, dat\'ee du 29 juin 2012.\footnote{\url{https://blogs.ethz.ch/kowalski/?s=isogenies}} Celui-ci, intitul\'e ``A bijective challenge'', commen\c cait comme suit :\\

\selectlanguage{english}
\it Étienne Fouvry, Philippe Michel and myself have just finished a new paper,\footnote{Pour les curieux, il s'agit de l'article \cite{Kowalski}.} which is available on my web page and will soon be also on arXiv. {\rm [\ldots]} Today I want to discuss a by-product that we found particularly nice (and amusing). It can be phrased as a rather elementary-looking challenge: given a prime number $p$, and an element $t$ of $\z/p\z$ which is neither 0, 4 or -4 modulo $p$, let $N_p(t)$ be the number of solutions $(x,y)\in (\z/p\z-\{0\})^2$ of the congruence
\[x+\frac{1}{x}+y+\frac{1}{y}+t=0.\]
The challenge is to prove, bijectively if possible, that
\[N_p(t)=N_p\left(\frac{16}{t}\right).\]
{\rm [\ldots]} This sounds simple and elegant enough that an elementary proof should exist, but our argument is a bit involved. First, the number $N_p(t)$ is the number of points modulo $p$ of the curve with equation above, whose projective (smooth) model is an elliptic curve, say $E_t$, over $\f_p$. Then we checked using Magma that $E_t$ and $E_{16/t}$ are isogenous over $\f_p$, and this is well-known to imply that the two curves have the same number of points modulo $p$.

{\rm [\ldots]} The best explanation of this has probably to do with the relation between the family of elliptic curves and the modular curve $Y_0(8)$ (a relation whose existence follows from Beauville’s classification of stable families of elliptic curves over $\mathbb{P}^1$ with four singular fibers, as C.~Hall pointed out), but we didn’t succeed in getting a proof of all our statements using that link. In fact, we almost expected to find the identities above already spelled out in some corner or other of the literature on modular curves and universal families of elliptic curvers thereon, but we did not find anything.\\

\rm
\selectlanguage{french}
Cette question de Fouvry, Kowalski et Michel a \'et\'e reprise pour l’\'edition 2014 du \tfjm, ainsi que pour sa version internationale, l’ITYM. Bien entendu, on ne peut pas lancer des \'el\`eves de lyc\'ee directement sur une question ouverte comme celle-ci. C’est pourquoi l’\'enonc\'e a \'et\'e modifi\'e de la fa\c con suivante pour amener les \'el\`eves \`a la question de ces chercheurs sans pour autant la leur poser frontalement. Voici l’\'enonc\'e pr\'ecis de ce probl\`eme lors du \tfjm :\\

\it Soit $p$ un nombre premier impair. On consid\`ere l’ensemble $\f_p=\z/p\z$ des entiers modulo $p$. On note $\f_p^*$ l’ensemble des \'el\'ements inversibles de cet ensemble, {\it i.e.}~les $x\in\f_p$ tels qu’il existe $y\in\f_p$ v\'erifiant $xy\equiv 1\pmod p$. Pour $x\in \f_p^*$, on note $\frac{1}{x}$ un tel \'el\'ement $y$ (remarquer qu’il est unique dans $\f_p$).

Dans ce probl\`eme, on veut \'etudier l’\'equation
\[x+\frac{1}{x}+y+\frac{1}{y}=t,\]
o\`u $x,y\in\f_p^*$ et $t\in\f_p$. On note $N_p(t)$ le nombre de solutions en $x$ et $y$ pour un $t$ fix\'e.
\begin{enumerate}
\item
Quel est le cardinal de l’ensemble $A:=\{x+\frac{1}{x}|x\in\f_p^*\}$ ? Et celui des ensembles
\[A+A:=\{a+b\,|\,a,b\in A\}\quad\text{et}\quad A\cdot A=\{a\cdot b\,|\,a,b\in A\}\, ?\]
On pourra commencer par les calculer explicitement pour des petites valeurs de $p$.
\item Calculer $N_p(0)$.
\item Pour des petites valeurs de $p$, calculer explicitement $N_p(t)$ pour tout $t$.
\item Existe-t-il des couples $(a,b)$ avec $a\in\f_p^*$ et $b\in\f_p$ tels que $N_p(t)=N_p(at+b)$ pour tout $t\in\f_p$ ? Existe-t-il $a\in\f_p^*$ tel que $N_p(t)=N_p(a\frac{1}{t})$ pour tout $t\in\f_p^*$ ? \'Etudier d’autres transformations possibles de $t$ laissant stable la valeur de $N_p$.
\item G\'en\'eraliser ces r\'esultats, \`a $\z/n\z$ pour $n$ quelconque (on pourra commencer avec les puissances des nombres premiers), puis aux corps finis quelconques.\\
\end{enumerate}

\rm Comme on pourra le remarquer, on s’est permis de poser des questions {\it a priori} plus difficiles que celle pos\'ee par Fouvry-Kowalski-Michel (qui est cach\'ee dans la question 4). Ceci est un facteur commun de tous les probl\`emes pos\'es au \tfjm, le but \'etant que le probl\`eme soit ouvert m\^eme pour les jurys qui \'evalueront les \'el\`eves lors du tournoi (ces jurys sont compos\'es de chercheurs, doctorants et \'etudiants de master ou de grandes \'ecoles).\\

Ce qui fait l’int\'er\^et de ce probl\`eme en particulier, c’est qu’il est issu d’une question pos\'ee par des vrais chercheurs en math\'ematiques \`a d'autres chercheurs, donc toute solution devient {\it ipso facto} un r\'esultat de recherche (certes pas de la plus haute importance, mais de la vraie recherche tout de m\^eme). Et c’est en mettant ensemble les diff\'erentes solutions des \'el\`eves qui ont particip\'e au \tfjm\, 2014 que l’on a r\'eussi \`a trouver une r\'eponse \'el\'ementaire \`a la question de Fouvry-Kowalski-Michel. Gr\^ace au \tfjm, des lyc\'eens fran\c cais ont donc fait un v\'eritable petit apport \`a la recherche math\'ematique. On pr\'esente leurs arguments dans ce qui suit.

\section{La preuve}

\subsection{\'Enonc\'e du th\'eor\`eme et r\'esultats pr\'eliminaires}\label{section preliminaires}

Dans toute la suite $p$ désigne un nombre premier impair. Pour tout $t \in \mathbb{F}_p$, on note $M_p(t)$ le nombre de solutions dans $\mathbb{F}_p^*$ de l'équation
\begin{equation}\label{unevariable}
x+\frac{1}{x}=t,
\end{equation}
et $N_p(t)$ le nombre de solutions dans $(\mathbb{F}_p^*)^2$ de l'équation
\begin{equation}\label{deuxvariables}
x+\frac{1}{x}+y+\frac{1}{y}=t.
\end{equation}

Notre but est de montrer le résultat suivant.

\begin{thm}\label{thm}
Pour tout $t \in \mathbb{F}_p^*$, on a $N_p \left( \frac{16}{t} \right)=N_p(t)$.
\end{thm}

La preuve du r\'esultat consiste en un comptage plus ou moins explicite des solutions. On \'etablit d'abord une borne sup\'erieure sur les quantit\'es $N_p(t)$, puis on montre que les quantit\'es $N_p \left( \frac{16}{t} \right)$ et $N_p(t)$ sont congrues modulo $p$. Ayant toutes ces informations, un simple argument de parit\'e conclut l'affaire.\\

On commence par donner une formule relativement explicite pour $N_p(t)$ : on a
\begin{equation} \label{formuleN}
N_{p}(t)=\sum_{a\in \mathbb{F}_p} M_{p}(a) M_{p}(t-a).
\end{equation}
On cherche donc \`a calculer explicitement les $M_p(a)$.

L'équation \eqref{unevariable} pour $t=a$ est équivalente à $x^{2}-ax+1=0$, qui est une équation quadratique de discriminant $a^2-4$. Elle a donc une solution si $a^2-4=0$, deux si $a^2-4$ est un carré et aucune si $a^2-4$ n'est pas un carré. On rappelle donc la d\'efinition du symbole de Legendre :
\[\left( \frac{a}{p} \right)=\begin{cases}
0 &\qquad\text{si } a\equiv 0\pmod p,\\
+1 &\qquad\text{si } a \text{ est un carr\'e non nul modulo } p,\\
-1 &\qquad\text{sinon.} 
\end{cases}\]
Sachant que le groupe multiplicatif $\f_p^*$ est cyclique d'ordre $p-1$, ce symbole peut aussi \^etre calcul\'e par ce qu'on appelle le crit\`ere d'Euler :
\[\left( \frac{a}{p} \right)\equiv a^{\frac{p-1}{2}}\pmod p.\]
\`A l'aide de ce symbole, on voit que l'on peut exprimer les valeurs de $M_p$ comme suit :
\begin{equation}\label{formuleM}
M_p(a)=1+\left( \frac{a^2-4}{p} \right).
\end{equation}

Ces formules nous permettent de prouver deux r\'esultats élémentaires sur les valeurs $N_p(t)$. Le premier est la borne sup\'erieure annonc\'ee plus haut, le deuxi\`eme concerne la parit\'e de ces valeurs, ce qui sera utile \`a la fin de la preuve.

\begin{lem} \label{majoration}
Pour tout $t \in \mathbb{F}_p$ on a
\[N_p(t) \leq N_p(0)=2p-4.\]
\end{lem}

\begin{proof}
En appliquant l'inégalité de Cauchy-Schwarz à \eqref{formuleN} on obtient
\begin{align*}
N_p(t) &\leq \sqrt{\left( \sum_{a\in \mathbb{F}_p} M_p(a)^2 \right) \left( \sum_{a\in \mathbb{F}_p} M_p(t-a)^2 \right)}\\
&=\sum_{a \in \mathbb{F}_p} M_p(a)^2=\sum_{a \in \mathbb{F}_p} M_p(a) M_p(-a)=N_p(0),
\end{align*}
où $M_p(a)=M_p(-a)$ d'après \eqref{formuleM}.

Pour calculer $N_p(0)$, on étudie la répartition des valeurs de $M_p$. D'une part, la formule \eqref{formuleM} nous dit que $M_p(2)=M_p(-2)=1$ et $M_p$ vaut $0$ ou $2$ partout ailleurs. D'autre part on a
\[ \sum_{a \in \mathbb{F}_p} M_p(a)=p-1,\] donc $M_p$ prend $\frac{p-1}{2}$ fois la valeur $0$, deux fois la valeur $1$ et $\frac{p-3}{2}$ fois la valeur $2$. On obtient donc
\[N_p(0)=\sum_{a \in \mathbb{F}_p} M_p(a)^2=2\cdot 1+\frac{p-3}{2}\cdot 4=2p-4.\]
\end{proof}

\begin{lem} \label{parite}
$N_p(t)$ est impair si et seulement si $t \equiv \pm 4 \pmod{p}$.
\end{lem}

\begin{proof}
Le cas $t=0$ découle immédiatement du lemme \ref{majoration}. Supposons $t \ne 0$ et soit $a \in \mathbb{F}_p$ : si $M_p(a) M_p(t-a)$ est impair, alors $a, t-a \in \{-2,2\}$ d'apr\`es la formule \eqref{formuleM}, donc $(a,t)$ vaut $(2,4)$ ou $(-2,-4)$. Ainsi, si $t \notin \{-4,4\}$ alors tous les termes de la somme \eqref{formuleN} sont pairs donc $N_p(t)$ est pair. Si $t \in \{-4,4\}$ alors le terme correspondant à $a=\frac{t}{2}$ est impair et tous les autres sont pairs donc $N_p(t)$ est impair.
\end{proof}

\subsection{Un argument inspir\'e de la preuve de Chevalley-Warning}\label{section chevalley-warning}

Le th\'eor\`eme de Chevalley-Warning est un r\'esultat classique de la th\'eorie des corps finis.

\begin{thm}[Chevalley-Warning]
Soit $\f$ un corps fini de caract\'eristique $p$ et soit $P(x_1,\ldots,x_n)$ un polyn\^ome de degr\'e $d$ \`a coefficients dans $\f$. Alors, si $n>d$, le nombre de solutions de l'\'equation $P=0$ est congru \`a $0$ modulo $p$.

En particulier, si $(0,\ldots,0)$ est une solution (par exemple, si $P$ est homog\`ene), alors il existe une solution non triviale.
\end{thm}

La preuve de ce th\'eor\`eme, qui est bien classique (cf.~par exemple le premier chapitre de \cite{SerreArithmetique}), consiste en une description astucieuse de l'ensemble des solutions qui utilise le petit th\'eor\`eme de Fermat. La m\^eme technique nous permettra de d\'emontrer dans la suite le r\'esultat suivant :

\begin{pro}\label{Np(t) congru a Np(16/t)}
Pour tout $t \in \mathbb{F}_p^*$ on a $N_p \left( \frac{16}{t} \right) \equiv N_p(t) \pmod{p}$.
\end{pro}

\begin{proof}
L'équation \eqref{deuxvariables} est équivalente à
\begin{equation}\label{2var-deg3}
x^2y+xy^2+x+y-txy=0,\quad\text{avec}\quad (x,y) \in (\mathbb{F}_p^*)^2,
\end{equation}
donc $(x^2y+xy^2+x+y-txy)^{p-1}$ vaut $0$ si $(x,y)$ est solution et $1$ sinon, d'apr\`es le petit th\'eor\`eme de Fermat. On peut ainsi compter les couples non-solutions modulo $p$ :
\[(p-1)^2-N_p(t) \equiv \sum_{x, y \in \mathbb{F}_p^*} (x^2y+xy^2+x+y-txy)^{p-1} \pmod{p}.\]
En développant le polynôme et en remplaçant $t$ par $-t$ (la bijection \'evidente $(x,y)\mapsto (-x,-y)$ donne $N_p(t)=N_p(-t)$), on obtient que
\[N_p(t)\equiv 1-\sum_{x, y \in \mathbb{F}_p^*}\sum_{\sum k_i=p-1} \frac{(p-1)!}{k_1!...k_5!} (x^2y)^{k_1}(xy^2)^{k_2} x^{k_3} y^{k_4} (txy)^{k_5} \pmod p.\]
Et en rappelant que $(p-1)!\equiv -1\pmod p$ d'apr\`es le th\'eor\`eme de Wilson, on obtient
\begin{align*}
N_p(t) &\equiv 1+\sum_{\sum k_i=p-1}  \frac{t^{k_5}}{k_1!...k_5!} \sum_{x, y \in \mathbb{F}_p^*}  x^{2k_1+k_2+k_3+k_5} y^{k_1+2k_2+k_4+k_5},\\
&\equiv 1+\sum_{\sum k_i=p-1}  \frac{t^{k_5}}{k_1!...k_5!} \sum_{x, y \in \mathbb{F}_p^*} x^{p-1+k_1-k_4} y^{p-1+k_2-k_3},\\
&\equiv 1+\sum_{\sum k_i=p-1}  \frac{t^{k_5}}{k_1!...k_5!} \left( \sum_{x \in \mathbb{F}_p^*} x^{k_1-k_4} \right) \left( \sum_{y \in \mathbb{F}_p^*} y^{k_2-k_3} \right)\pmod p.
\end{align*}

On a maintenant besoin du lemme suivant, cl\'e dans la preuve de Chevalley-Warning.

\begin{lem}
Soit $a \in \mathbb{Z}$. Alors $\sum_{x \in \mathbb{F}_p^*} x^a$ vaut $-1$ si $p-1$ divise $a$, et $0$ sinon.
\end{lem}
\begin{proof}
Si $p-1$ divise $a$, chacun des $p-1$ termes de la somme vaut $1$ d'apr\`es le petit th\'eor\`eme de Fermat, d'où le résultat. Sinon, soit $g$ un générateur du groupe multiplicatif $\mathbb{F}_p^*$. On a
\[\sum_{x \in \mathbb{F}_p^*} x^a=\sum_{i=0}^{p-2} g^{ia}=\frac{g^{a(p-1)}-1}{g^a-1}=0.\]
\end{proof}

D'après le lemme, si le terme correspondant à $(k_1,...,k_5)$ est non nul, alors $k_1-k_4$ et $k_2-k_3$ sont divisibles par $p-1$, donc valent $0$, $p-1$ ou $1-p$. Si par exemple $k_1-k_4=p-1$, alors $k_1=p-1$ et $k_2=...=k_5=0$, et la contribution de ce terme vaut $\frac{1}{(p-1)!}(-1)(-1)=-1$, et de même lorsque $k_2$, $k_3$ ou $k_4$ vaut $p-1$. Dans tous les autres termes non nuls, on a $k_1=k_4$ et $k_2=k_3$ d'où
\begin{equation}\label{k1k2k5}
N_p(t) \equiv -3+\sum_{2k_1+2k_2+k_5=p-1} \frac{t^{k_5}}{k_1 !^2 k_2!^2 k_5!} \pmod{p}.
\end{equation}

\begin{rem}
C'est à partir d'ici que notre preuve diffère de celle du théorème de Chevalley-Warning : alors que l'hypothèse sur le degré dans Chevalley-Warning garantit que tous les termes non triviaux sont nuls, il reste ici un certain nombre de termes à contrôler, le degré de \eqref{2var-deg3} étant plus élevé que le nombre de variables.
\end{rem}

On cherche donc à regrouper les termes de \eqref{k1k2k5} selon la valeur de $k_5$. On constate que, $p-1$ étant pair, $k_5$ est toujours pair. Soit donc $k_5$ pair dans $[\![0, p-1]\!]$ fixé. On a
\begin{align}\label{somme a k5 fixe}
\begin{split}
\sum_{2k_1+2k_2=p-1-k_5} \frac{t^{k_5}}{k_1 !^2 k_2!^2 k_5!} &= \frac{t^{k_5}}{k_5! {\left( \frac{p-1-k_5}{2} \right)!}^2} \sum_{k_1+k_2=\frac{p-1-k_5}{2}} \binom{\frac{p-1-k_5}{2}}{k_1}^2\\
&=  \frac{t^{k_5}}{k_5! {\left( \frac{p-1-k_5}{2} \right)!}^2} \binom{p-1-k_5}{\frac{p-1-k_5}{2}},
\end{split}
\end{align}
où la dernière ligne utilise l'égalité $\sum_{k=0}^n \binom{n}{k}^2=\binom{2n}{n}$. C'est un cas particulier de l'identité de Vandermonde, qu'on peut obtenir par exemple par double-comptage.

Donc, en appliquant \eqref{k1k2k5} et \eqref{somme a k5 fixe} respectivement à $t$ et $\frac{16}{t}$, on a d'une part
\[N_p(t) \equiv -3+\underset{k\text{ pair}}{\sum_{k=0}^{p-1}} \frac{1}{k! {\left( \frac{p-1-k}{2} \right)!}^2} \binom{p-1-k}{\frac{p-1-k}{2}} t^k \pmod{p},\]
et d'autre part
\[N_p \left( \frac{16}{t} \right) \equiv -3+\underset{k\text{ pair}}{\sum_{k=0}^{p-1}} \frac{1}{k! {\left( \frac{p-1-k}{2} \right)!}^2} \binom{p-1-k}{\frac{p-1-k}{2}} \left( \frac{16}{t} \right)^{k} \pmod{p},\]
ce qui donne, en \'echangeant $k$ et $p-1-k$,
\[N_p \left( \frac{16}{t} \right) \equiv -3+\underset{k\text{ pair}}{\sum_{k=0}^{p-1}} \frac{1}{(p-1-k)! {\left( \frac{k}{2} \right)!}^2} \binom{k}{\frac{k}{2}} \frac{t^k}{16^k} \pmod{p}.\]

Pour en déduire $N_p \left( \frac{16}{t} \right) \equiv N_p(t) \pmod{p}$, il suffit donc de démontrer que pour tout $k$ pair dans $[\![0, p-1]\!]$ on a
\[\frac{1}{k! {\left(\frac{p-1-k}{2} \right)!}^2} \binom{p-1-k}{\frac{p-1-k}{2}} \equiv \frac{1}{(p-1-k)! {\left( \frac{k}{2} \right)!}^2} \binom{k}{\frac{k}{2}} \frac{1}{16^k} \pmod{p},\]
ou encore, en chassant les d\'enominateurs et en prenant des racines carr\'ees,
\[(p-1-k)! {\left( \frac{k}{2} \right)!}^2 4^k \equiv (-1)^{\frac{p-1}{2}} k! {\left( \frac{p-1-k}{2} \right)!}^2 \pmod{p},\]
ce qui se fait directement par r\'ecurrence sur $k$.
\end{proof}

\subsection{Le coup de grâce}\label{section coup de grace}

D'après la proposition \ref{Np(t) congru a Np(16/t)}, on a $N_p \left( \frac{16}{t} \right) \equiv N_p(t) \pmod{p}$, donc $N_p \left( \frac{16}{t} \right) - N_p(t)$ est divisible par $p$. De plus, on sait que $N_p(t)$ et $N_p \left( \frac{16}{t} \right)$ sont positifs et majorés par $N_p(0)=2p-4$ d'après le lemme \ref{majoration}, donc leur différence ne peut valoir que $-p$, $0$ ou $p$. Enfin, $N_p(t)$ est impair si et seulement si $t$ vaut $-4$ ou $4$ d'après le lemme \ref{parite}. Les quantités $N_p(t)$ et $N_p \left( \frac{16}{t} \right)$ ont alors la même parité et leur différence doit donc valoir $0$, d'où le résultat.

\section{Quelques commentaires sur la preuve}
Nous sommes partis des solutions pr\'esent\'ees par les deux \'equipes fran\c caises qui participaient au tournoi international, dont l'une \'etait encadr\'ee par le premier auteur. Ces \'equipes \'etaient donc les meilleures \'equipes du \tfjm\, 2014, et elles ont eu l'occasion d'échanger avec les autres participants à la compétition nationale. Cela fait de ces solutions un effort commun de plusieurs \'el\`eves de tr\`es haut niveau. 

Ainsi, les r\'esultats de la section \ref{section preliminaires} sont enti\`erement d\^us aux participants du tournoi. Ils ont \'et\'e modifi\'es dans ce texte pour gagner en brievet\'e et r\'epondre sp\'ecifiquement \`a l'\'enonc\'e original de Fouvry-Kowalski-Michel. Comme il a \'et\'e d\'emontr\'e dans la section \ref{section coup de grace}, ces r\'esultats suffisent pour r\'eduire la preuve de l'\'egalit\'e en une preuve de congruence modulo $p$. C'est l\`a que l'exp\'erience a jou\'e un r\^ole, puisque c'est leur encadrant ({\it i.e.}~le premier auteur) qui a vu que les id\'ees appliqu\'ees dans Chevalley-Warning pouvaient \^etre utiles pour d\'emontrer une telle affirmation. La section \ref{section chevalley-warning} a \'et\'e donc r\'edig\'ee par ses soins.

Il est cependant important de remarquer que le niveau des techniques utilis\'ees dans cette d\'emonstration reste tr\`es basique, y compris dans la section \ref{section chevalley-warning}. C'est pourquoi on ne doute aucunement qu'elle aurait \'et\'e \`a la port\'ee des \'el\`eves si jamais leur encadrant leur avait sugg\'er\'e, tel un directeur de th\`ese, d'aller lire la d\'emonstration de Chavelley-Warning pour y trouver des pistes. Malheureusement, cette id\'ee n'est venue \`a l'esprit du premier auteur qu'apr\`es le tournoi international.\\

Enfin, nous tenons \`a mentionner que, malgr\'e le fait d'avoir trouv\'e une preuve \'el\'ementaire \`a la question de Fouvry-Kowalski-Michel, le d\'efi n'est pas pour autant compl\`etement relev\'e. En effet, ces chercheurs demandent une preuve non seulement \'el\'ementaire, mais \emph{bijective} de l'\'enonc\'e du th\'eor\`eme \ref{thm}, alors que notre preuve \'etablit seulement que les deux ensembles ont le m\^eme cardinal. D'autres preuves non-bijectives ont d'ailleurs \'et\'e trouv\'ees avant celle du \tfjm, mais toujours avec des outils plus complexes. Il y a notamment :
\begin{itemize}
\item la preuve mentionn\'ee sur le blog de Kowalski, bas\'ee sur des calculs informatiques sur des courbes elliptiques ;
\item une preuve (non publi\'ee) par David Zywina, bas\'ee sur l'\'etude de certaines formes modulaires, comme il \'etait sugg\'er\'e \`a la fin du post de Kowalski ;
\item un r\'esultat encore plus g\'en\'eral sur les repr\'esentations $\ell$-adiques du groupe fondamental d'une courbe, d\^u \`a Deligne et Flicker (cf.~\cite{DeligneFlicker}), qui implique aussi l'\'enonc\'e.
\end{itemize}
Apr\`es les courbes elliptiques, les formes modulaires, les repr\'esentations $\ell$-adiques et Chevalley-Warning, \`a vous de trouver une meilleure preuve !

\appendix

\section{Quelques mots sur le \tfjm\\ (par Matthieu Lequesne et Giancarlo Lucchini Arteche)}

\subsection{Qu'est-ce que le \tfjm ?}
Le \textit{Tournoi Fran\c cais des Jeunes Math\'ematiciennes et Math\'ematiciens} (\tfjm) est une comp\'etition de math\'ematiques qui propose \`a des \'el\`eves de lyc\'ee de travailler par \'equipe sur une s\'erie de probl\`emes ouverts pendant plusieurs mois. Il est organis\'e par l'association Animath et par une grande \'equipe d'anciens participants, aujourd'hui tous en \'etudes sup\'erieures.\\

\subsubsection{Un peu d'histoire}
C'est David Zmiaikou, jeune math\'ematicien bi\'elorusse, qui est \`a l'origine du \tfjm, ainsi que de sa version internationale, l'ITYM (International Tournament of Young Mathematicians). Pendant ses ann\'ees en tant que doctorant \`a l'universit\'e Paris-Sud, il a eu l'id\'ee d'organiser un tournoi international \`a l'image des tournois de math\'ematiques en Bi\'elorussie. Avec l'aide de Bernardo da Costa, \'egalement doctorant \`a la facult\'e d'Orsay, et avec le soutien de l'association Animath, le premier tournoi international des jeunes math\'ematiciens a eu lieu en 2009 dans les installations de l'universit\'e avec la participation de six \'equipes provenant de quatre pays diff\'erents : Bi\'elorussie, Bulgarie, France et Russie.\\

En 2011, deux ans apr\`es le succ\`es du premier tournoi, un troisi\`eme doctorant, Igor Kortchemski, s'est ajout\'e \`a ce groupe pour cr\'eer la version fran\c caise du tournoi, en partie pour mieux pr\'eparer les \'equipes fran\c caises envoy\'ees \`a l'ITYM. C'est ainsi que le \tfjm\, est n\'e, accueillant pour sa premi\`ere \'edition quatre \'equipes venues de Paris, Versailles et Strasbourg.

Le \tfjm\, s'est bien d\'ev\'elopp\'e depuis. En effet, il accueillait d\'ej\`a dix-huit \'equipes en 2014 et, pour l'\'edition 2015, il a fait un grand saut en proposant pour la premi\`ere fois des tournois r\'egionaux aux quatre coins de la France.\\

\subsubsection{Le format}
Le \tfjm\, se compose aujourd'hui de deux \'etapes : des tournois r\'egionaux r\'epartis partout en France, puis une finale nationale \`a l'\'Ecole polytechnique et \`a l'ENSTA Paristech. Une liste contenant entre 8 et 12 probl\`emes est publi\'ee chaque ann\'ee en d\'ecembre. Ainsi les \'el\`eves peuvent regarder les probl\`emes pendant leurs vacances et commencer \`a travailler pendant le mois de janvier. Ils sont encadr\'es pour cela par des professeurs de lyc\'ee, mais aussi par des \'etudiants, doctorants ou chercheurs des universit\'es voisines. En 2017, des tournois r\'egionaux sont pr\'evus dans les villes de Lille, Lyon, Paris, Rennes, Strasbourg et Toulouse lors d'un week-end pendant le mois d'avril. La finale nationale aura lieu pour sa part au mois de mai.\\

Lors des diff\'erents tournois, les \'equipes s'affrontent par poules sous la forme de ``joutes math\'ematiques''. Dans celles-ci, chaque \'equipe occupe \`a tour de r\^ole trois fonctions :
\begin{itemize}
\item \emph{pr\'esenter la solution} d'un probl\`eme sur lequel ils ont travaill\'e ;
\item \emph{critiquer la solution} qui a \'et\'e pr\'esent\'ee : chercher les impr\'ecisions, demander des explications et lancer un d\'ebat constructif qui a lieu au tableau ;
\item \emph{\'evaluer le d\'ebat} et \'eventuellement le prolonger s'il y reste des choses int\'eressantes \`a discuter ;
\end{itemize}
le tout devant un jury form\'e de chercheurs, professeurs, doctorants et \'etudiants de master ou de grandes \'ecoles.

Outre leur prestation orale, les \'el\`eves sont aussi incit\'es \`a \'evaluer les solutions \'ecrites par leurs camarades. Ils r\'edigent pour cela des notes de synth\`ese sur la solution qui va \^etre pr\'esent\'ee. Celles-ci sont \'egalement \'evalu\'ees par le jury.

Les solutions sont soumises une semaine avant chaque tournoi. Peu apr\`es, un tirage au sort a lieu pendant lequel se d\'ecide la formation des poules, les probl\`emes qui seront pr\'esent\'es et l'ordre des pr\'esentations.\\

\subsection{Une rencontre avec la recherche}
On voit ainsi que, \`a la diff\'erence des math\'ematiques que les participants rencontrent au lyc\'ee et dans d'autres comp\'etitions, ce tournoi leur propose une vraie rencontre avec la recherche en math\'ematiques. En effet :

\begin{enumerate}
\item Les probl\`emes propos\'es sont compl\`etement ouverts puisque personne n'en conna\^it la solution g\'en\'erale, pas m\^eme leurs auteurs. Les \'el\`eves doivent donc chercher par eux-m\^emes en ayant droit \`a toutes les ressources possibles, y compris des r\'esultats interm\'ediaires trouv\'es dans des livres, voire sur internet. En cela, ils ont la m\^eme libert\'e qu'un chercheur dans son laboratoire face \`a ses probl\`emes math\'ematiques quotidiens.

\item Le travail se fait par \'equipes. On peut alors travailler \`a plusieurs sur un m\^eme probl\`eme, en r\'epartissant les t\^aches en fonction des comp\'etences de chacun. On peut se relire les uns les autres et on peut aussi travailler sur plusieurs probl\`emes \`a la fois. Les probl\`emes recouvrant divers domaines des math\'ematiques, l'\'equipe enti\`ere simule ainsi un ``mini-laboratoire'' de recherche !

\item Le r\^ole de l'encadrant est \`a mettre en parall\`ele avec celui d'un directeur de th\`ese. En effet, il est l\`a pour aider les \'el\`eves lorsqu'ils se trouvent dans une impasse, pour leur donner des outils qui les aideront \`a trouver la solution et pour leur donner des pistes de recherche, sans toutefois leur donner directement la solution.

\item Les d\'ebats qui ont lieu au tableau lors du tournoi n'ont rien \`a envier aux discussions que les chercheurs ont tous les jours dans un laboratoire de math\'ematiques. En effet, il s'agit de quelqu'un qui pr\'esente ses r\'esultats de mani\`ere concise et de son interlocuteur qui lui fait remarquer l\`a o\`u il a \'et\'e impr\'ecis, tout comme un chercheur qui pr\'esente ses travaux au coll\`egue du bureau voisin pour \^etre s\^ur que tout semble correct.

\item Le travail de relecture des solutions des autres \'equipes est similaire au travail des rapporteurs d'articles de recherche. De plus, le fait de lire d'autres solutions aide les \'el\`eves \`a comprendre ce qui a pu manquer en p\'edagogie dans leurs propres textes. C'est \'egalement dans cette dynamique, en lisant de plus en plus d'articles, qu'un chercheur apprend petit \`a petit \`a en \'ecrire de meilleurs.
\end{enumerate}

Le tournoi permet aussi les \'echanges entre \'el\`eves et chercheurs de premier niveau, lesquels participent de fa\c con b\'en\'evole en tant que jurys. Les pr\'esidents de jury interagissent notamment avec les \'el\`eves apr\`es le premier tour de chaque tournoi pour faire un compte-rendu des d\'ebats afin de les faire progresser dans la pr\'esentation de r\'esultats scientifiques. En outre, depuis 2013, les \'el\`eves assistent \`a des expos\'es donn\'es par des chercheurs lors des c\'er\'emonies d'ouverture ou de cl\^oture.\\

\subsection{Les participants}
Bien que le r\^eve des organisateurs du \tfjm\, soit de donner une vraie exp\'erience de recherche \`a tous les \'el\`eves de lyc\'ee, il faut avouer que ce tournoi s'adresse aux meilleurs et plus motiv\'es d'entre eux. En effet, afin de pouvoir mener \`a bien une r\'eflexion pouss\'ee il est n\'ecessaire que ces \'el\`eves soient suffisamment \`a l'aise avec les connaissances math\'ematiques que requiert le programme de lyc\'ee. Le format des \'enonc\'es est d\'ej\`a un d\'efi important en soi d\^u \`a l'usage d'un langage math\'ematique qui est parfois un peu lourd, m\^eme si tout est fait pour simplifier au maximum les concepts et ne pas alourdir les notations. Mais cela fait parti du d\'efi : ils ont du temps pour dig\'erer ces nouveaux concepts, d\'efinitions et notations, \'eventuellement avec l'aide de leurs encadrants. D'ailleurs, ce sont souvent ces probl\`emes qui sont les plus plebiscit\'es par les \'el\`eves car ceux qui ont un \'enonc\'e simple sont souvent les plus difficiles, tout comme en recherche !\\

Le tournoi s'adresse donc \`a tous les lyc\'eens de premi\`ere ou terminale scientifique ayant un bon niveau en math\'ematiques mais surtout une grande curiosit\'e et la volont\'e d'aller chercher un peu plus loin que le programme scolaire. Ainsi, dans chaque lyc\'ee on pourrait constituer une \'equipe avec les meilleurs \'el\`eves de premi\`ere ou terminale, encadr\'es par un ou deux professeurs. On notera que l'on a parfois vu des participants plus jeunes (en seconde, voire au coll\`ege) faire \'equipe avec des camarades plus \^ag\'es, et \'egalement des \'el\`eves qui n'\'etaient pas en fili\`ere scientifique mais n\'eanmoins tr\`es int\'eress\'es par les math\'ematiques. 

Parfois, certains \'etablissements poss\`edent d\'ej\`a un club de math\'ematiques anim\'e par un enseignant et ce sont les \'el\`eves du club qui forment une \'equipe. Le \tfjm\, fournit ainsi \`a ces clubs du contenu pour travailler pendant toute la premi\`ere moiti\'e du deuxi\`eme s\'emestre. Ces lyc\'eens se r\'eunissent ensuite une fois par semaine avec les professeurs pour \'echanger sur leurs avanc\'ees, et \'echangent entre eux par mail, pendant les trois mois qui pr\'ec\`edent le tournoi (la fr\'equence peut s'intensifier \`a l'approche du tournoi).\\

Ainsi, le public du tournoi ne se restreint pas aux \'etablissements les plus r\'eput\'es, bien que ceux-ci restent encore sur-repr\'esent\'es. Il existe de nombreux exemples d'\'equipes d'\'etablissements pas particuli\`erement r\'eput\'es qui ont r\'ealis\'e de tr\`es bonnes performances. L'ouverture des tournois r\'egionaux a permis d'accueillir un public plus large et nous souhaitons continuer dans ce sens. Certaines \'equipes plus faibles s'attaquent aux probl\`emes avec leurs propres moyens, souvent moins th\'eoriques et plus exp\'erimentaux, mais vont parfois bien loin dans le tournoi. Cela est d\^u notamment au fait que les \'el\`eves s'approprient tr\`es bien les quelques concepts qu'ils ont pu traiter et ont d\'evelopp\'e de bonnes capacit\'es p\'edagogiques mises en valeur lors du d\'ebat oral.

Le principal facteur limitant reste certainement l'investissement horaire important que le tournoi demande de la part des encadrants. Ce sont en effet des professeurs qui d\'ecident d'encadrer b\'en\'evolement leurs \'el\`eves et de sacrifier plusieurs soir\'ees et week-ends, souvent sans reconnaissance de leur hi\'erarchie.\\

\subsection{Le \tfjm\, en chiffres}

\subsubsection{Un tournoi en pleine croissance}
Apr\`es un premier tournoi en 2011 avec seulement 17 participants, le \tfjm\, a atteint un cap en 2014 avec 106 participants. Ceci a motiv\'e la cr\'eation des tournois r\'egionaux en 2015, lesquels rassemblent depuis plus de 150 participants. Le nombre de tournois n'a cess\'e d'augmenter : quatre tournois r\'egionaux en 2015, puis cinq en 2016 et six en 2017. L'objectif \`a long terme serait de parvenir \`a en organiser un par acad\'emie.\\

\subsubsection{Un tournoi ouvert \`a tous}
L'\'edition 2016 du tournoi, avec ses cinq tournois r\'egionaux, a r\'euni des participants scolaris\'es dans 16 acad\'emies distinctes, de Nantes \`a Montpellier, en passant par Nancy-Metz. Ceci est possible gr\^ace au format du tournoi, qui permet de travailler en temps libre et \`a distance. D'ailleurs, des \'equipes mixtes inter-d\'epartamentales participent chaque ann\'ee.

Le \tfjm\, se veut \'egalement de plus en plus mixte avec une part de math\'ematiciennes toujours croissante : de 12\% de filles en 2012, ce pourcentage est en augmentation constante et a atteint 23\% en 2016 (l'ann\'ee 2011 n'est pas repr\'esentative compte tenu du nombre total de participants).

\begin{center}
\includegraphics[width=.4\textwidth]{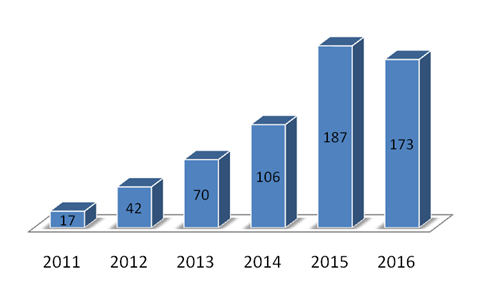}
\hspace{1cm} \includegraphics[width=.4\textwidth]{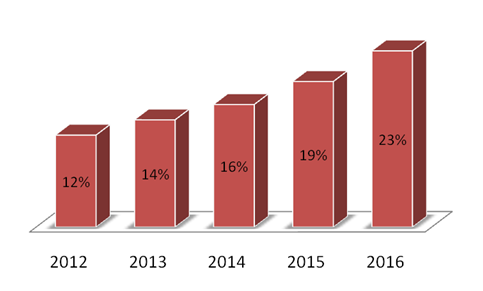}\\
Nombre total de participants et pourcentage de participantes par ann\'ee.
\end{center}

\end{document}